\documentclass[12pt , reqno, fleqn]{amsart}
\usepackage[margin=1in]{geometry}
\usepackage{amssymb}
\usepackage{amsthm}
\usepackage{amstext}
\usepackage{amsbsy}
\usepackage{amscd}
\usepackage{enumerate}
\usepackage{chngpage}
\usepackage{mathtools}
\usepackage{amsmath}
\usepackage{hyperref}
\usepackage{tikz}
\usepackage{stmaryrd}
\usepackage{color}

\newtheorem{thm}{Theorem}[section]
\newtheorem{lemma}[thm]{Lemma} \newtheorem{cor}[thm]{Corollary}
\newtheorem{prop}[thm]{Proposition}
\theoremstyle{definition}
\newtheorem{defn}[thm]{Definition}
\newtheorem{ques}[thm]{Question}

\newtheorem{example}[thm]{Example}

\newtheorem{rmk}[thm]{Remark}

\newcommand{\Z}{\mathbb Z}

\newcommand{\F}{\mathbb F}

\newcommand{\Fq}{\mathbb{F}_q}

\DeclareMathOperator{\PGL}{PGL}
\newcommand{\mb}[1]{\mathbb{#1}}
\newcommand{\mc}[1]{\mathcal{#1}}

\begin{document}
\title{An algorithm for counting arcs in higher-dimensional projective space}
\author{Kelly Isham}
\date{}
\address{Department of Mathematics\\ 214 McGregory Hall\\
	Colgate University\\
Hamilton, NY}
\email{kisham@colgate.edu}
\maketitle

\begin{abstract}
	An $n$-arc in $(k-1)$-dimensional projective space is a set of $n$ points so that no $k$ lie on a hyperplane. In 1988, Glynn gave a formula to count $n$-arcs in the projective plane in terms of a relatively small number of combinatorial objects called superfigurations. Several authors have used this formula to count $n$-arcs in the projective plane for $n \le 10$. In this paper, we determine a formula to count $n$-arcs in projective 3-space. We then use this formula to give exact expressions for the number of $n$-arcs in $\mb{P}^3(\F_q)$ for $n \le 7$, which are polynomial in $q$ for $n \le 6$ and quasipolynomial in $q$ for $n=7$. Lastly, we generalize to higher-dimensional projective space.
\end{abstract}

\section{Introduction}\label{intro_sec}
Let $k \le n$. An $n$-arc in $(k-1)$-dimensional projective space is a set of $n$ points so that no $k$ lie on a hyperplane. Specializing to $k=3$, an $n$-arc in the projective plane is a set of $n$ points so that no 3 lie on a line. 

Arcs in $\mb{P}^2(\Fq)$ are closely related to several objects of interest. First, we can identify an $n$-arc with a $k\times n$ generator matrix with entries in $\Fq$ whose columns are given by some choice of affine representative for each point in the $n$-arc. Since no $k$ of these points lie on a hyperplane, no $k \times k$ minor of the generator matrix vanishes. By this association, $n$-arcs are also related to `maximum distance separable' (MDS) codes, which are linear codes for which the Singleton bound is achieved. Finally, an $n$-arc can be identified with an $\F_q$-point on the open subset of the Grassmannian $G(k,n)$ for which all Pl\"ucker coordinates are nonzero. See \cite{hirschfeldb, hirschfelda} for more on these connections. Significantly, any information about one of these objects immediately gives new results about the others. 

In 1955, Segre \cite{segre} highlighted three questions about arcs, including a question about determining the largest size of an arc in $(k-1)$-dimensional projective space. In a projective plane of order $q$, the answer is known -- if $q$ is odd, the largest size of an arc in $q+1$ and if $q$ is even, the largest size is $q+2$. When $k > 3$ and $q\ge k$, the MDS Conjecture -- a famous conjecture in coding theory -- states that the largest size of an arc should be $q+1$. Researchers have been making progress on this problem; see \cite{balllavrauw} for a recent survey on large arcs.

In this paper, we will discuss a counting version of Segre's question. Let $C_{n,k}(q)$ denote the number of ordered $n$-arcs in $\mb{P}^{k-1}(\Fq)$.  A major question about arcs is the following.

\begin{ques}
	For fixed $n$ and $k$, what is $C_{n,k}(q)$ as a function of $q$?
\end{ques} 

Let $M_{n,k}(q)$ denote the number of $[n,k]$ MDS codes over $\F_q$. Let $U_{n,k}(q)$ denote the open stratum of the Grassmannian $G(k,n)$ over $\F_q$ for which all Pl\"ucker coordinates are nonzero. It is known that $M_{n,k}(q) = \#U_{n,k}(q)$, see e.g. \cite{hirschfeldthas} for more details. The following proposition is an easy generalization of \cite[Lemma 2]{iss}.
\begin{prop}\label{mds_arc}
	Fix positive $k,n \in \Z$. Then
	$$
	M_{n,k}(q)=\#U_{n,k}(q) = \frac{(q-1)^n}{|\PGL_k(\F_q)|} C_{n,k}(q).
	$$
\end{prop}

This proposition demonstrates the connection between arcs, MDS codes, and rational points on the Grassmannian. Any statement about $C_{n,k}(q)$ in this paper can be converted to a statement about $M_{n,k}(q)$ or $\#U_{k,n}(q)$ using Proposition \ref{mds_arc}. 

\subsection{Arcs in the projective plane}
When $k \ge 4$, there is a unique projective space of order $q$ up to isomorphism, namely $\mb{P}^{k-1}(\F_q)$. However, when $k=3$, there can be several non-isomorphic projective planes of order $q$. In this setting, we use the notation $C_{n}(\Pi)$ where $\Pi$ is some projective plane of order $q$. In \cite{glynn}, Glynn produces an algorithm to count the number of $n$-arcs in any projective plane of order $q$ in terms of simpler combinatorial objects. This algorithm has been used to determine exact formulas for $C_{n,3}(\Pi)$ when $n \leq 9$. Glynn finds that $C_{n,3}(\Pi)$ is polynomial in $q$ when $n\leq 6$.

A function $f$ is \emph{quasipolynomial} if there exists finitely many polynomials $g_0, \ldots, g_{N-1}$ so that $f= g_i$ whenever $q \equiv i \pmod{N}$. The function $C_{n,3}(\Pi)$ is quasipolynomial when $n \in \{7,8,9\}$ \cite{glynn, iss, kklpw}. Iampolskaia, Skorobogatov, and Sorokin \cite{iss} count [9,3] MDS codes and derive their formula for $C_{9,3}(\mb{P}^2(\Fq))$ as a corollary. Kaplan, Kimport, Lawrence, Peilen, and Weinreich \cite{kklpw} extend their work to count $9$-arcs in any projective plane of order $q$.

\begin{defn}\cite{bb}
	A \emph{linear space} $(\mc{P}, \mc{L})$ is a pair of sets where $\mc{P}$ denotes a set of points and $\mc{L}$ denotes a set of lines that satisfies the following properties:
	\begin{enumerate}
		\item Every line in $\mc{L}$ is a subset of $\mc{P}$.
		\item Any two distinct points belong to exactly one line in $\mc{L}$.
		\item Every line in $\mc{L}$ contains at least 2 points.
	\end{enumerate}
\end{defn}

\begin{defn}
	Two linear spaces $f= (\mc{P}_1, \mc{L}_1)$ and $g = (\mc{P}_2, \mc{L}_2)$ are isomorphic if there exists a bijection $\mc{P}_1 \rightarrow \mc{P}_2$ that preserves lines.
\end{defn}

Since every two points determine a line, we call a line containing at least three points a \emph{full line}. A planar space $(\mc{P}, \mc{L})$ is uniquely determined by $\mc{P}$ and the set of full lines. Thus we can refer to a linear space by its set of points and full lines only. From now on, we only consider full lines and we will drop the word ``full."

\begin{defn}\cite{kklpw}
	A \emph{superfiguration} in the projective plane is a linear space so that every line contains at least 3 points and every point lies on at least 3 lines.
\end{defn}

We give an alternate definition of superfigurations which will be useful later on. Let the \emph{index} of a point be the number of (full) lines through that point. A \emph{superfiguration} in the projective plane is a linear space so that the index of every point is at least 3.

These superfigurations, which are highly symmetric and contain many lines, are important objects in classical projective geometry. The Fano plane is the unique superfiguration on $7$ points, and the M\"obius-Kantor configuration is the unique superfiguration on 8 points. The Hesse superfiguration is one of ten superfigurations on 9 points. It contains 9 points and 12 lines and can be realized by the 9 inflection points of a complex smooth cubic curve.

Let $\Pi$ be a projective plane of order $q$. A \emph{strong realization} in $\Pi$ of a superfiguration $s$ is an embedding of the points $\mc{P}$ into $\Pi$ so that no extra collinearities are formed. We let $A_s(\Pi)$ denote the number of strong realizations of $s$ in $\Pi$. We now state Glynn's Theorem for the number of $n$-arcs in the projective plane.

\begin{thm}\cite{glynn}\label{glynn_thm}
	There exist polynomials $p(q)$ and $p_s(q)$ such that for any projective plane $\Pi$ of order $q$, 
	$$
	C_n(\Pi) = p(q) + \sum_{s} p_s(q) A_s(\Pi)
	$$
	where the summation is taken over all superfigurations $s$ on at most $n$ points.
\end{thm}

\begin{rmk}\label{glynn_rmk}
	Consider counting ordered $n$-arcs in $\mb{P}^{k-1}(\Fq)$ by counting $k \times n$ generator matrices with the property that no $k \times k$ minor vanishes. By the Inclusion-Exclusion Principle, we could determine the number of such matrices by counting $k \times n$ matrices for which at least one maximal minor vanishes. Fix an ordering on the $\binom{n}{k}$ maximal minors. For each $(i_1, \ldots, i_{\binom{n}{k}}) \in \{0,1\}^{\binom{n}{k}}$, we must determine the number of $k \times n$ matrices with entries in $\Fq$ for which minor $M_{i_j}$ vanishes if $i_j=0$ and does not vanish if $i_j =1$. There are $2^{\binom{n}{k}}-1$ such patterns of minors to consider. 
	Theorem \ref{glynn_thm} is important because it reduces the number of objects to consider significantly. While an exact formula for the number of superfigurations on at most $n$ points is not known,  there are far fewer than $2^{\binom{n}{3}}-1$ of them. For example when $n=7$, there are $2^{35}-1$ patterns of minors to consider, yet only one superfiguration $s$ on 7 points up to isomorphism. There are 168 superfigurations in the isomorphism class of $s$. Table \ref{superfig_table} gives the number of linear spaces and superfigurations up to isomorphism for $7 \le n \le 12$.
	
	\begin{table}[!htp] \caption{Number of linear spaces and superfigurations on $n$ points up to isomorphism\cite{betten, kklpw}}\label{superfig_table}
		\begin{tabular}{c|cccccc}
			\hline
			$n$ &7&8&9&10&11&12\\
			\hline
			Linear spaces & 24&69&384&5,250&232,929&28,872,973\\
			\hline
			Superfigurations&1&1&10&151&16,234&$>179,000$\\
			\hline
		\end{tabular}
	\end{table}
\end{rmk}

The summation in Theorem \ref{glynn_thm} is over all superfigurations on at most $n$ points. However, since $A_s(\Pi) = A_t(\Pi)$ whenever $s$ is isomorphic to $t$, we can modify this theorem to sum over all isomorphism classes of superfigurations on at most $n$ points instead. 

In forthcoming joint work, we modify Glynn's formula to make computations simpler and we use this modified algorithm to show that the number of 10-arcs in $\mb{P}^2(\mb{F}_q)$ is a nonquasipolynomial function in $q$. While no explicit 10-arc formula is given, we show that the formula depends on the Fourier coefficients of certain modular forms which have models that are elliptic curves or singular K3 surfaces. We then conjecture that the number of $n$-arcs will continue to be nonquasipolynomial for larger $n$, as the number of $n$-arcs in the projective plane should follow Mn\"ev's Universality Theorem \cite{mnev}. However, we cannot prove this conjecture without explicitly determining all pieces that appear in Theorem \ref{glynn_thm}, which becomes computationally infeasible when $n > 10$.  The common obstruction to proving these types of theorems is that we cannot guarantee `bad' pieces do not cancel out. For examples of this obstruction occurring in other problems, see \cite{paksoffer,vaughanlee}.

\subsection{Arcs in Projective 3-Space}
Based on the difficulty of computation for 10-arcs, it seems infeasible to count the number of 11-arcs in $\mb{P}^2(\mb{F}_q)$. Further, combining the results from \cite{glynn, iss, kklpw} and the forthcoming work on $C_{10,3}(q)$ gives the transitions from polynomial to quasipolynomial to nonquasipolynomial. Instead, we take a new direction in the study of $n$-arcs. In this paper, we generalize Glynn's formula by producing an algorithm to count the number of $n$-arcs in $\mb{P}^3(\mb{F}_q)$. We also outline how to adapt these ideas to count $n$-arcs in $\mb{P}^{k-1}(\Fq)$ where $k > 4$. We begin by setting up the terminology that we will need later on.

In 2-dimensional space, the basic geometric objects are points and lines. In 3-dimensional space, we must consider points, lines, and planes.

\begin{defn}
A \emph{planar space} is a triple of sets $(\mc{P}, \mc{L}, \mc{H})$ where $\mc{P}$ is the set of points, $\mc{L}$ is the set of lines, and $\mc{H}$ is the set of planes such that 
\begin{enumerate}
	\item $\mc{L}, \mc{H} \subseteq 2^{\mc{P}}$
	\item $(\mc{P}, \mc{L})$ is a linear space
	\item Any three distinct non-collinear points lie on a unique plane.
\end{enumerate}
\end{defn}

We use the notation $\mc{H}$ to represent planes since planes in 3-dimensional projective space are the same as hyperplanes. Planar spaces are very general spaces. For example, for all $k \ge 4$, the $(k-1)$-dimensional projective and affine spaces are planar spaces.

Two planar spaces $(\mc{P}_1, \mc{L}_1, \mc{H}_1)$ and $(\mc{P}_2, \mc{L}_2, \mc{H}_2)$ are isomorphic if there exists a bijection $\mc{P}_1 \rightarrow \mc{P}_2$ that preserves lines and planes.

\begin{rmk}
	The number of planar spaces on $n$ points is equal to the number of non-isomorphic simple matroids on a set of $n$ points with rank at most 4. Adding columns from Table 4 in \cite{mmib} leads to Table \ref{psf_count}.

	\begin{table}[!htp]\caption{Number of planar spaces on $n$ points up to isomorphism \cite{mmib}}\label{psf_count}
			\begin{center}
		\begin{tabular}{c|ccccccccc}
			\hline
		$n$ & 2&3&4&5&6&7&8&9&10\\
		\hline
		Planar spaces & 1&2&4&8&21&73&686&186,365&4,884,579,115\\	
		\hline	
		\end{tabular}
	\end{center}
	\end{table}

\end{rmk}

\begin{defn}
For a planar space $f = (\mc{P}, \mc{L}, \mc{H})$, a strong realization of $f$ in $\mb{P}^3(\F_q)$ is an injective mapping $\sigma:\mc{P} \rightarrow \mb{P}^3(\F_q)$ such that each subset $Q$ of $\mc{P}$
\begin{enumerate} \item is contained in a line of $f$ if and only if $\sigma(Q)$ is contained in a line of $\mb{P}^3(\F_q)$ and 
	\item is contained in a plane of $f$ if and only if $\sigma(Q)$ is contained in a plane of $\mb{P}^3(\F_q)$. 
	\end{enumerate}
	For any planar space $f$, let $A_f(4,q)$ be the number of strong realizations of $f$. 
\end{defn}

We use the notation $A_f(4,q)$ to avoid confusion with the notation for the number of strong realizations in the projective plane given in \cite{kklpw}. The use of the numeral 4 indicates that we are considering embeddings of points into $\mb{P}^3(\F_q)$.

A \emph{full line} of $f$ is a line containing at least 3 points and a \emph{full plane} of $f$ is a plane containing at least 4 points. 

\begin{rmk}
	From now on, we use the terms line and plane to mean full line and full plane respectively. Abusing notation, in our examples we will only write down the full lines in $\mc{L}$ and full planes in $\mc{H}$. For example, we can define a planar space on four points with lines given by the set $\{\{1,2,3\}, \{1,4\}, \{2,4\},\{3,4\}\}$ and planes given by $\{\{1,2,3,4\}\}$. However, we would simply write $\mc{L} = \{\{1,2,3\}\}$ and $\mc{H} = \{\{1,2,3,4\}\}$ as a planar space is uniquely determined by its full lines and full planes.
	\end{rmk}

A point has \emph{index} $(i,j)$ if it lies on exactly $i$ (full) planes and $j$ (full) lines.
\begin{defn}
	\label{hyperfig_def}
A \emph{hyperfiguration} is a planar space on $n$ points such that for every point $P$, the index $(i,j)$ of $P$ satisfies at least one of the following:
\begin{enumerate}
	\item $i \ge 4$
	\item $j \ge 3$
	\item $(i,j) = (3,0)$.
\end{enumerate}
\end{defn}

This definition is a bit surprising as it does not appear to be the direct generalization of a superfiguration. In fact, omitting conditions (2) and (3) gives the most direct generalization of a superfiguration, namely that every plane contains at least four points and every point lies on at least four planes. Taking conditions 1 and 2 together allows for subplanes of a planar space to contain isomorphic copies of superfigurations. Thus condition (2) makes sense to include as superfigurations were special objects in projective planes, so they should often be considered special objects in projective 3-space. We call index $(3,0)$ a \emph{surprising index} since it is not obvious why we must allow this case in the definition of hyperfiguration. This will be made clear in Section \ref{gen_sec}.

Skorobogatov \cite{skorobogatov_matroid} studies similar formulas for the number of representations of a matroid over $\Fq$. His formula is in terms of a summation over matroids that are \emph{special} and \emph{co-special}; see \cite{skorobogatov_matroid} for these definitions. He also gives a necessary criterion for a matroid to be special. It is likely that Definition \ref{hyperfig_def} exactly classifies the matroids of rank at most 4 that are both special and co-special.
\subsection{Main Results}
Theorem \ref{glynn_thm} gives the count for $n$-arcs in the projective plane in terms of realizations of superfigurations, which informally are combinatorial objects that contain many lines. In this paper, we generalize Theorem \ref{glynn_thm} to 3-dimensional projective space. We do so by showing that $C_{n,4}(q)$ can be expressed in terms of a linear combination of the number of strong realizations for hyperfigurations, which are combinatorial objects that contain either many lines or many planes.

\begin{thm}\label{p3count}
	There exist polynomials $p(q)$ and $p_h(q)$ in $\Z[q]$ such that
	$$
	C_{n,4}(q) = p(q) + \sum_{h} p_h(q) A_h(4,q)
	$$
	where the summation runs over all isomorphism classes $h$ of hyperfigurations on at most $n$ points. Moreover, there is an algorithm that produces $p(q)$ and $p_h(q)$ for each isomorphism class $h$.
\end{thm}

We emphasize here that Theorem \ref{p3count} significantly reduces the number of objects to consider when compared to Inclusion-Exclusion. The data in Table \ref{psf_count} demonstrates that the number of planar spaces on $n$ points is significantly smaller than $2^{\binom{n}{4}}-1$. We give the number of hyperfigurations up to isomorphism for small $n$ in Table \ref{hyp_table}.
	
	\begin{table}[!htp]\caption{Number of hyperfigurations on $n$ points up to isomorphism}\label{hyp_table}
		\begin{center}
			\begin{tabular}{c|ccc}
				\hline
				$n$ &6 & 7 & 8\\
				\hline
				Hyperfigurations&1&6&235\\
				\hline
			\end{tabular}
		\end{center}
	\end{table}

It is also interesting to note that it is not obvious why the summation in Theorem \ref{p3count} is over hyperfigurations, as these are not a direct generalization of superfigurations. In Section \ref{gen_sec} we explain the subtleties that make hyperfigurations the right object to choose. Throughout this paper, we abuse notation and refer to an isomorphism class of a hyperfiguration as a hyperfiguration.

We then implement the algorithm given in the proof of Theorem \ref{p3count} in Sage \cite{sage} to express $C_{n,4}(q)$ for $4 \le n \le 7$. 

\begin{thm}\label{arcform} Let $a(q) = \begin{cases} 1& 2 \mid q \\ 0 & 2\nmid q\end{cases}.$ Then
	\begin{align*}
		C_{4,4}(q) &=  (q^2 + q + 1)(q^2 + 1)(q + 1)^2q^6\\[1em]
		C_{5,4}(q) &= (q^2 + q + 1)(q^2 + 1)(q + 1)^2(q - 1)^3q^6\\[1em]
		C_{6,4}(q) &= (q^2 + q + 1)(q^2 + 1)(q + 1)^2(q - 1)^3(q-2)(q-3)(q-4)q^6\\[1em]
		C_{7,4}(q) &= (q^2 + q + 1)(q^2 + 1)(q + 1)^2(q-1)^3q^6\bigg(q^6 - 28q^5 \\
		&+ 323q^4 - 1952q^3 + 6462q^2 - 11004q + 7470-30a(q)\bigg).
	\end{align*}
\end{thm}

When $n = 4,5,$ and 6, $C_{n,4}(q)$ can also be determined by counting methods. We will describe these in Section \ref{count_sec}. By using the duality between $[n,k]$ MDS codes and $[n,n-k]$ MDS codes, one can determine the number of 7-arcs in $\mb{P}^3(\Fq)$ from Glynn's \cite{glynn} formula for 7-arcs in $\mb{P}^2(\Fq)$. Our algorithm gives another way of producing $C_{n,4}(q)$ when $4 \le n \le 7$. Importantly, this algorithm still works for $n \ge 8$, meaning that it is now more feasible to compute the number of $n$-arcs in $\mb{P}^3(\Fq)$ for larger $n$.

\subsection{Outline}
In Section \ref{gen_sec}, we prove Theorem \ref{p3count} and give an algorithm for counting $n$-arcs in $\mb{P}^3(\Fq)$. In Section \ref{count_sec}, we use the algorithm from Section \ref{gen_sec} to determine $C_{n,4}(q)$ for $4\le n \le 7$. In Section \ref{higher_dim_sec}, we discuss an approach to generalizing hyperfigurations in higher-dimensional projective space. We then prove that a formula to compute $C_{n,k}(q)$ in terms of these realizations of these generalized hyperfigurations exists for all $k \ge 4$.

\section{Generalizing Glynn's Theorem for $n$-Arcs in $\mb{P}^3(\Fq)$}\label{gen_sec}
We can define a partial order on planar spaces on $n$ points as follows. Let $\mc{P} = \{1, 2, \ldots, n\}$. Suppose $f = (\mc{P}, \mc{L}_1 , \mc{H}_1)$ and $g  = (\mc{P}, \mc{L}_2, \mc{H}_2)$ are two planar spaces on $n$ points. Then $g \ge f$ if each line in $\mc{L}_1$ is contained in some line of $\mc{L}_2$ and each plane in $\mc{H}_1$ is contained in some plane of $\mc{H}_2$.

\begin{example}
	Let $f$ be the planar space on five points with $\mc{L}_1 = \{\{0,1,2\}\}$ and $\mc{H}_1 = \{\{0,1,2,3\}, \{0,1,2,4\}\}$. Let $g$ be the planar space on five points with $\mc{L}_2 = \{\{0,1,2,3\}\}$ and $\mc{H}_2 = \{\{0,1,2,3,4\}\}$. Then $g \ge f$.
	
	If we take $h$ to be the planar space on five points with $\mc{L}_3 = \emptyset$ and $\mc{H}_3 = \{\{0,1,2,3,4\}\}$, then $g \ge h$, but $h$ is not comparable to $f$.
\end{example}

\begin{defn}
	For a planar space $f = (\mc{P}, \mc{L}, \mc{H})$, a weak realization of $f$ in $\mb{P}^3(\F_q)$ is an injective mapping $\tau:\mc{P} \rightarrow \mb{P}^3(\F_q)$ such that for every subset $Q \subseteq \mc{P}$
	\begin{enumerate} \item if $Q$ is contained in a line in $\mc{L}$, then $\tau(Q)$ is contained in a line in $\mb{P}^3(\F_q)$ and 
		\item if $Q$ is contained in a plane in $\mc{H}$, then $\tau(Q)$ is contained in a plane of $\mb{P}^3(\F_q)$. 
	\end{enumerate}
	For any planar space $f$, let $B_f(4,q)$ be the number of weak realizations of $f$. 
\end{defn}

In other words, a weak realization of $f$ is an injective mapping $\mc{P} \rightarrow \mb{P}^3(\Fq)$ so that all lines in $\mc{L}$ and all planes in $\mc{H}$ are preserved, but extra collinearities or coplanarities may be imposed. From these definitions, we see that
$$
B_f(4,q) = \sum_{g \geq f} A_g(4,q).
$$

We are ready to state the main lemma, which is a generalization of \cite[Lemma 2.10]{kklpw}. The idea is that we can rewrite the number of weak realizations of a planar space $f$ on $n$ points in terms of a $\Z[q]$-linear combination of the number of strong realizations of a planar space on $n-1$ points.
\begin{lemma}
\label{mainlemma}
Suppose that a planar space $f$ on $n \geq 4$ points has a point of index $(i,j)$ where $i< 4$ and $j<3$ and $(i,j) \ne (3,0)$. Then we have
$$
B_f(4,q) = \sum_{g \geq f'} \mu(f,g) A_g(4,q)
$$
where $f'$ is the planar space obtained from removing the point of index $(i,j)$ from $f$ and $\mu(f,g)$ is a polynomial in $q$. 
\end{lemma}

\begin{proof}
Let $f$ be a planar space and let $m$ be a point of index $(i,j)$ for which $i < 4$ and $j < 3$ and $(i,j) \ne (3,0)$. Reorder the points in $f$ so that $m$ is the last point. Let $f'$ be the planar space obtained from removing the point $m$ of index $(i,j)$ from $f$. We can form any weak realization of $f$ in $\mb{P}^3(\mb{F}_q)$ by taking a strong realization of $g \geq f'$ and adding back the point $m$. Observe that adding a point to $g$ will give a weak realization of $f$ since extra collinearities or coplanarities may be formed. For each $g$, $\mu(f,g)$ is the number of ways to add a point to a strong realization of $g$ to get a weak realization of $f$. In order to prove this lemma, we must show that $\mu(f,g)$ is a polynomial in $\Z[q]$ for every $g \ge f'$. In order to do this, we work by cases depending on the index.\\

For each $g$, let $P_g$ be a set of $n-1$ ordered points in $\mb{P}^3(\Fq)$ that form a strong realization of $g$. In this proof, we now work in $\mb{P}^3(\Fq)$ rather than considering planar spaces abstractly. Thus the points, lines, and planes in $g$ must satisfy all properties of finite projective 3-space over $\Fq$. For example, two distinct planes must intersect at a line. See \cite[page 126]{moorhouse} for the axioms of $\mb{P}^3(\Fq)$.\\

\noindent\textbf{Index (0,0)} Suppose we remove point $m$ from $f$ to get $f'$. Let $g \ge f'$. We must add a point to $g$ to obtain a weak realization of $f$. Since point $m$ is not contained in any lines or planes of $f$, we can simply choose any remaining point to get a weak realization of $f$. Therefore
$$
\mu(f,g) = (q^3 + q^2 + q + 1) - (n-1).
$$
\vspace{.05in}

\noindent\textbf{Index (0,1)} Let $L'$ be the line in $f'$ corresponding to the line in $f$ that contained $m$. Extend this line $L'$ to the line $L_g$ in $g$. Adding any point of $L_g$ not already in $P_g$ gives a weak realization of $f$. Thus
$$
\mu(f,g) = (q+1) -\# (P_g\cap L_g).
$$
\vspace{.05in}

\noindent\textbf{Index (0,2)} It is impossible for $f$ to have a point of index $(0,2)$ since any two intersecting full lines in $\mb{P}^3(\Fq)$ are contained in a full plane. 

\vspace{.2in}

\noindent\textbf{Index (1,0)} Let $H$ be the plane in $f$ containing $m$ and let $H'$ be the corresponding subset of points in $f'$.  If $H'$ is contained in some line of $g$, then adding any point not in $g$ gives a weak realization of $f$. Thus
$$
\mu(f,g) = (q^3+q^2+q+1) - (n-1).
$$

Otherwise, extend $H'$ to the plane $H_g$ in $g \geq f'$. We can add any point to $H_g$ that is not already in the realization of $g$. Thus
$$
\mu(f,g) =
(q^2 +q +1) - \#(P_g \cap H_g).
$$
\vspace{.05in}

\noindent\textbf{Index (1,1)} The point $m$ is contained in a plane $H$ and a line $L$ in the planar space $f$. The line $L$ must be contained in the plane $H$. If not, then take a point $r \ne m$ in $H$ that does not lie on $L$. The plane $\{r\} \cup L$ is a plane in $f$ containing $m$ that is distinct from $H$. But this implies that $m$ does not have index $(1,1)$.

Let $H'$ and $L'$ be the subsets of points in $f'$ corresponding to $H$ and $L$ after removing point $m$. By the above argument, $L' \subset H'$. Extend $L'$ to the line $L_g$ in $g$. We claim that 
$$
\mu(f,g) = (q+1) - \#(P_g \cap L_g).
$$

If $H'$ is contained in a line $L$ in $g$, then it is enough to add the point $m$ to the line $L_g$.
Otherwise, extend $H'$ to the plane $H_g$. Since $L' \subset H$, then $L_g \subset H_g$. Adding a point to $L_g$ also adds a point to the plane $H_g$.
\vspace{.2in}

\noindent\textbf{Index (1,2)} Let $H'$ be the subset of points in $f'$ corresponding to the plane $H$ in $f$ containing $m$. Let $L_1'$ and $L_2'$ be the lines in $f'$ corresponding to the lines $L_1$ and $L_2$ in $f$ containing $m$. Since $L_1 \cap L_2 \ne \emptyset$ and both lines intersect $H$, then $L_1$ and $L_2$ are contained in $H$. Thus $L_1', L_2' \subseteq H'$. Extend $L_1'$ and $L_2'$ to the lines $L_{1,g}$ and $L_{2,g}$ of $g$. If $L_{1,g}$ and $L_{2,g}$ are the same line, then proceed as in case $(1,1)$. 

Otherwise, they must be distinct lines. Suppose that $H'$ is contained in a line $L_g$ of $g$. Then $L_{1,g}$ and $L_{2,g}$ are contained in $L_g$, so $L_{1,g}$ and $L_{2,g}$ are not distinct lines.

Lastly, suppose $H'$ extends to a plane $H_g$ of $g$. Since $L_1'$ and $L_2'$ are contained in $H'$, then $L_{1,g}, L_{2,g} \subseteq H_g$. Thus we simply add the intersection point $r$ of $L_{1,g}$ and $L_{2,g}$ giving
 $$
 \mu(f,g) =
\begin{cases}
0 & r \in P_g\\
1 & r \not \in P_g.
\end{cases} 
 $$
\vspace{.1in}

\noindent\textbf{Index (2,0)} Let $H_1'$ and $H_2'$ be the subsets of points in $f'$ corresponding to the planes containing $m$ in $f$. If $H_1'$ and $H_2'$ are both contained in lines of $g,$ then adding any point gives a weak realization of $f$, so
$$\mu(f,g) = (q^3+q^2+q+1) - (n-1).$$

Now suppose (without loss of generality) that $H_1'$ is contained in a line $L_{1,g}$, but $H_2'$ is not contained in any line of $g$. Extend $H_2'$ to the plane $H_{2,g}$. Adding any point $r$ to $H_{2,g}$ forms the plane $L_{1,g} \cup \{r\}$. Therefore
$$
\mu(f,g) = (q^2+q+1) - \#(P_g \cap H_{2,g}).
$$

We have taken care of all cases for which at least one of $H_1'$ and $H_2'$ is contained in a line of $g.$ Suppose that neither $H_1'$ nor $H_2'$ are contained in a line of $g$. Extend $H_1'$ and $H_2'$ to the planes $H_{1,g}$ and $H_{2,g}$ respectively. If the points in $H_{1,g}$ and $H_{2,g}$ are contained in a single plane $H_g$, then proceed as in case $(1,0)$.

Otherwise, recall that two planes intersect at a line, call it $L_g$. We have
$$
\mu(f,g)= (q+1) - \#(L_g \cap P_g).
$$
\vspace{.05in}

\noindent\textbf{Index (2,1)} Let $H_1'$ and $H_2'$ be the subsets of points in $f'$ corresponding to the planes containing $m$ in $f$. Let $L'$ be the subset of points in $f'$ corresponding to the line containing $m$. Extend $L'$ to $L_g$ in $g$. Observe that a point can have index (2,1) if and only if $L = H_1 \cap H_2$. Thus $L' = H_1' \cap H_2'$. If $H_1'$ and $H_2'$ are contained in the same line or plane in $g$, we proceed as in case (1,1).

We claim that in all other cases, 
$$
\mu(f,g)= (q+1) - \#(P_g\cap L_g).
$$

Observe that $H_1'$ and $H_2'$ cannot be contained in different lines $L_{1,g}$ and $L_{2,g}$ respectively since $L'$ is contained in both $H_1'$ and $H_2'$ and we are assuming $H_1'$ and $H_2'$ are distinct.

Now suppose without loss of generality that $H_1'$ is contained in a line of $g$, but $H_2'$ is not. Extend $H_2'$ to the plane $H_{2,g}$ in $g$. In order to add a point to both planes, we can simply add a point to $H_{2,g}$. This point must also lie on the line $L_g$. Since $L_g \subset H_{2,g}$, then it suffices to add a point to $L_g$. 

Now suppose that $H_1'$ and $H_2'$ extend to the distinct planes $H_{1,g}$ and $H_{2,g}$ respectively. Since $L' = H_1' \cap H_2'$, then $L_g = H_{1,g} \cap H_{2,g}$. Thus adding a point to $L_g$ adds a point to $H_{1,g}$ and $H_{2,g}$ as well.
\vspace{.2in}

\noindent\textbf{Index (2,2)} This case is impossible. Suppose the point $m$ lies on two planes $H_1 = \{a_1, \ldots, a_s, m\}$ and $H_2 = \{ b_1, \ldots, b_t,m\}$ and two lines $L_1$ and $L_2$. Clearly one of these lines, say $L_1$, must be $H_1 \cap H_2$. Since $L_2 \ne L_1$, then the line $L_2$ is contained in exactly one of $H_1$ or $H_2$. Suppose without loss of generality that $L_2 \subset H_1$. Take some point $b_i \in H_2$ not on $L_2$. Then $L_2 \cup \{b_i\}$ forms a plane containing the point $m$ that is distinct from $H_1$ and $H_2$. Thus the point $m$ does not have index (2,2). 
\vspace{.2in}

\noindent\textbf{Index (3,1)}
Let $H_1', H_2'$, and $H_3'$ be the subsets of points in $f'$ corresponding to the planes in $f$ containing $m$. Let $L_1'$ be the line in $f'$ corresponding to the line in $f$ containing $m$. Extend $L_1'$ to $L_{g}$. Observe that the point $m$ has index (3,1) if and only if $L_1'=H_1' \cap H_2' \cap H_3'.$

If all three of $H_1', H_2'$, and $H_3'$ extend to the same line or plane in $g$, we proceed as in case (1,1). 

If two of $H_1', H_2'$, and $H_3'$ extend to the same line or plane, but the third does not, we proceed as in case (2,1).

Otherwise, all three extend to distinct lines or planes in $g$.

We claim that in all cases
$$
\mu(f,g) = (q+1) - \#(P_g \cap L_g).
$$

Observe that in order for $H_1', H_2'$, and $H_3'$ to extend to distinct objects, at most one of $H_j'$ can extend to a line since $L_1' \subset H_i'$ for $i=1,2,3$. 

Suppose that $H_1'$ extends to the line $L_{1,g}$ and that $H_2', H_3'$ extend to the planes $H_{2,g}, H_{3,g}$ respectively. We must add a point to $L = H_{2,g} \cap H_{3,g}$ and to $L_g$. Since $L_1'$ is contained in $H_2'$ and $H_3'$, then $L_g$ must be contained in both $H_{2,g}$ and $H_{3,g}$. Thus $L_g = L$, so we must add a point to $L_g$.

Lastly, suppose that $H_1', H_2'$, and $H_3'$ all extend to distinct planes $H_{1,g}, H_{2,g}$, and $H_{3,g}$ respectively. It suffices to add a point to $L_g \subset H_{i,g}$ for $i=1,2,3$.
\vspace{.2in}

\noindent\textbf{Index (3,2)}
Let $H_1', H_2',$ and $H_3'$ be the subsets of points in $f'$ corresponding to the planes containing $m$ in $f$. Let $L_1'$ and $L_2'$ be the lines in $f'$ corresponding to the lines $L_1$ and $L_2$ containing $m$ in $f$. Observe that $L_1$ and $L_2$ are coplanar in $f$, so $L_1'$ and $L_2'$ must be coplanar in $f'$.  In particular, $L_1'$ and $L_2'$ are contained in $H_i'$ for some $i \in \{1,2,3\}$. Further, $L_1$ and $L_2$ must be intersection lines between pairs of planes corresponding to $H_1', H_2'$, and $H_3'$. Extend $L_1'$ and $L_2'$ to the lines $L_{1,g}$ and $L_{2,g}$ in $g$.

If $L_1'$ and $L_2'$ are contained in the same line in $g$, let $i$ be the number of distinct lines or planes that extend $H_1', H_2'$, and $H_3'$. We proceed as in case $(i,1)$.

We claim that in all remaining cases, it suffices to add the intersection point $r$ of $L_{1,g}$ and $L_{2,g}$, so 
$$
\mu(f,g) = \begin{cases}
0 &r\in P_g\\
1 & r\not \in P_g.
\end{cases}
$$

Suppose first that $H_1', H_2',$ and $H_3'$ extend to lines in $g$.  This case is impossible since $L_{1,g}$ and $L_{2,g}$ must be contained in (and so equal to) one of these lines, but we are assuming $L_{1,g} \ne L_{2,g}$.

Next, suppose that $H_1'$ and $H_2'$ extend to lines $L_{1,g}^H$ and $L_{2,g}^H$ in $g$, but $H_3'$ extends to the plane $H_g$ in $g$. Recall that $L_1'$ and $L_2'$ lie in $H_i'$ for some $i \in \{1,2,3\}$. Since $L_{1,g}  \ne L_{2,g}$, then $L_1', L_2' \subseteq H_3'$, so $L_{1,g}$ and $L_{2,g}$ are contained in $H_g$. Adding the intersection point $r$ to $g$ creates a weak realization of $f$.

Suppose that $H_1'$ extends to a line $L$ in $g$, but $H_2'$ and $H_3'$ extend to planes $H_{2,g}$ and $H_{3,g}$ respectively. If $H_{2,g} = H_{3,g}$, then $L_{1,g}$ and $L_{2,g}$ must lie in this plane. Thus it suffices to add their intersection point. Otherwise, suppose the planes $H_{2,g}$ and $H_{3,g}$ intersect at a line $L$. By the observation at the beginning of this case, $L = L_{1,g}$ or $L = L_{2,g}$.

Lastly, suppose that $H_1', H_2'$, and $H_3'$ extend to planes $H_{1,g}, H_{2,g}$, and $H_{3,g}$ in $g$ respectively. If all three planes are the same, since $L_{1,g}$ and $L_{2,g}$ must lie on this plane, we can add their intersection point $r$ to get a weak realization of $f$. Suppose these three planes intersect at a line $L$. Then $L =L_{1,g} = L_{2,g}$. Since we are assuming $L_{1,g} \ne L_{2,g}$, this is impossible.

Otherwise, the three planes intersect at a point. By construction, this point must be the intersection point $r$ of $L_{1,g}$ and $L_{2,g}$.

\end{proof}

\begin{rmk}
	It may seem like this lemma does not take into account that lines can be skew in $\mb{P}^3(\Fq)$. However, if we have skew lines $L_1 = \{0,1,2\}$ and $L_2 = \{3,4,5\}$, then by the properties of planar spaces, $\{0,1,2,3\}, \{0,1,2,4\}$,
	$ \{0,1,2,5\}, \{0,3,4,5\}, \{1,3,4,5\}, \{2,3,4,5\}$ are all planes in $f$. Therefore the points included in these two lines all have index $(i,j)$ with $i \ge 4$ and so we do not attempt to remove any of the points on these skew lines. Thus, within the lemma, we can assume all lines must intersect.
\end{rmk}

\begin{rmk}
	We can now discuss Definition \ref{hyperfig_def}. Observe that if a plane in $f$ contains an isomorphic copy of a superfiguration, then $f$ should likely be considered a special object in projective 3-space. 
	
	Suppose $f$ has a point of index $(3,0)$ and suppose we were to remove it to obtain $f'$. Let $H_1', H_2',$ and $H_3'$ be the subsets of points corresponding to the planes in $f$ containing the point of index $(3,0)$. Suppose all three subsets extend to distinct planes $H_{1,g}, H_{2,g}, H_{3,g}$ in $g \ge f'$. Notice that there is ambiguity in how we should add a point to $g$. That is, we do not know whether $H_{1,g}, H_{2,g}$, and $H_{3,g}$ should intersect at a line or at a single point. Certainly if $|H_{1,g} \cap H_{2,g} \cap H_{3,g}| > 1$, we know these three planes intersect at a line; however, if $|H_{1,g} \cap H_{2,g} \cap H_{3,g}| \le 1$, we cannot tell what the intersection type of these planes should be. Since there is ambiguity, we must omit the case (3,0) from Lemma \ref{mainlemma}. Thus $(3,0)$ is included in the definition of hyperfiguration as a surprising index.
\end{rmk}

We can now prove Theorem \ref{p3count}.

\begin{proof}[Proof of Theorem \ref{p3count}]

First note that counting $n$-arcs in $\mb{P}^3(\mb{F}_q)$ is the same as counting sets of $n$ points such that no $4$ lie on a plane. Instead, we will determine $C_{n,4}(q)$ by counting all sets of $n$ points such that at least one set of 4 points forms a plane. Thus $C_{n,4}(q)$ is a linear combination of $A_f(4,q)$ for all planar spaces $f$ on at most $n$ points. We will show that we can simplify this formula by only considering hyperfigurations $h$ on at most $n$ points.

We work inductively on the number of points $m \le n$. We first find $A_f(4,q)$ and $B_f(4,q)$ for the unique planar space on 1 point. Observe that $$A_f(4,q) = B_f(4,q) = (q^3+q^2+q+1) + 0 \cdot B_{f'}(4,q)$$ for the unique planar space $f'$ on 0 points.

Assume that for all $f$ on $m$ points, we can express $A_f(4,q)$ as a $\Z[q]$-linear combination of $A_h(4,q)$ for all hyperfigurations $h$ on at most $m$ points. Fix $f$ on $m+1$ points. If $f$ is not a hyperfiguration, then
$$
A_f(4,q) = B_f(4,q) - \sum_{g > f} A_g(4,q).
$$ 
Use Lemma \ref{mainlemma} to write 
$$B_f(4,q) = \sum_{g \ge f'}  \mu(f,g) A_g(4,q).$$
By induction, we can express each $A_g(4,q)$ as a $\Z[q]$-linear combination of $A_h(4,q)$ for hyperfigurations $h$ on at most $m$ points. If $f$ is a hyperfiguration, we can simply write $A_f(4,q)$.\\

Continuing for all $f$ on $m+1$ points, we see that we can express all $A_f(4,q)$ as a $\Z[q]$-linear combination of $A_h(4,q)$ for hyperfigurations on at most $m+1$ points. By induction, we can continue until $m =n$.\\

To conclude, observe that if $f$ and $g$ are isomorphic planar spaces, then $A_f(4,q) = A_g(4,q)$ and $B_f(4,q) = B_g(4,q)$. 

\end{proof}

Observe that this proof gives an algorithm for counting arcs in $\mb{P}^3(\F_q)$.

\section{Formulas for $C_{n,4}(q)$}\label{count_sec}
In this section, we will prove Theorem \ref{arcform}. It is interesting to note that Kaipa \cite{kaipa} gives the first three main terms for $C_{n,k}(q)$. There is a typo in Kaipa's result that we correct below.

\begin{thm}\cite[Corollary 1.2]{kaipa}\label{kaipa_thm}
	Fix positive integers $n$ and $k$ so that $n > k$. Let $\delta = k(n-k)$, $N = \binom{n}{k}$, and 
	$$
	b_2(k,n) = \frac{N^2-5N+4}{2} -\frac{N\delta(\delta-n-3)}{2(\delta+n+1)} -(n-1)(N-n) -\frac{n^2-3n+2}{2}.
	$$
	For each fixed $n$, the number of $\PGL_k(\Fq)$-equivalence classes of $n$-arcs in $\mb{P}^{k-1}(\Fq)$ is asymptotically equal to
	$$
	q^{\delta -n+1 } -(N-n)q^{\delta-n} +b_2(k,n) q^{\delta-n-1} + O(q^{\delta -n-2}).
	$$
	
\end{thm}

We verify the first three main terms in $C_{n,4}(q)$ when $5 \le n \le 7$ by multiplying the formula in Theorem \ref{kaipa_thm} by $|\PGL_4(\Fq)|$.
\subsection{Verifying formulas for $n \le 6$}
The number of $n$-arcs in $\mb{P}^3(\F_q)$ for $n \le 5$ are simple to count by hand. 
When $n=4$, we choose any three non-collinear points, then select a point not on the plane formed by these three points. A 5-arc is a set of five points in general position. Thus the number of 5-arcs is equal to $|\PGL_4(\F_q)|$. There are no hyperfigurations on $n \le 5$ points, so our algorithm gives formulas that exactly match these counts.

We can also count 6-arcs combinatorially. Observe that any 6-arc determines a unique twisted cubic in $\mathbb{P}^3(\F_q)$. The group $\PGL_4(\F_q)$ acts on the set of twisted cubics. Moreover, under this action, all twisted cubics are projectively equivalent. Thus we can count the number of twisted cubics via the Orbit-Stabilizer Theorem. When $q \ge 5$, the stabilizer of a given twisted cubic is $\PGL_2(\F_q)$. See \cite{bdmp} for more details. Thus when $q \ge 5$, the number of twisted cubics is
$$
 \frac{|\PGL_4(\F_q)|}{|\PGL_2(\F_q)|}.
$$
Let $$P(q+1, 6) = \prod_{i=0}^5 (q+1-i)$$ be the number of ways of choosing six ordered points on the twisted cubic. 
Since we get a different ordered arc for each choice of six points on the twisted cubic, multiplying the previous formula by $P(q+1,6)$ and simplifying gives
\begin{equation}\label{c6eq}
C_{6,4}(q) = (q^2 + q + 1)(q^2 + 1)(q + 1)^2(q - 1)^3(q-2)(q-3)(q-4)q^6.
\end{equation}
Note that when $q < 5$, the number of 6-arcs in $\mb{P}^3(\F_q)$ is equal to 0. Thus \eqref{c6eq} holds for all prime powers $q > 0$. 

Next we verify that our algorithm gives the correct formula for $C_{6,4}(q)$.
\begin{prop} We have
	\begin{align*}
		C_{6,4}(q) &= q^{18} - 9q^{17} + 25q^{16} - 16q^{15} - 58q^{14}- 32q^{13} - 10q^{12} + 82q^{11} \\
		&+ 73q^{10} + 41q^9 - 15q^8 - 66q^7 - 16q^6 + 40A_6(4,q)
	\end{align*}
	where $A_6(4,q)$ is the number of strong realizations of the hyperfiguration on 6 points.
\end{prop}

When $n = 6$, there is exactly one hyperfiguration, which has full planes $\{0, 1, 2, 3\},$ $\{0, 1, 2, 4\}, \{0, 1, 2, 5\}, \{0, 3, 4, 5\},\{1, 3, 4, 5\}, \{2, 3,4 , 5\}$ and full lines given by the sets $\{0, 1, 2\}$ and $\{3,4,5\}$. These lines are necessarily skew. We simply count the number of strong realizations as follows. First select three points on a line. Then pick any point not on that line. These four points lie on a plane, so choose the fifth point to be any point not on this plane. Finally, pick a third point on the line formed by the fourth and fifth points. This gives $$A_6(4,q) = (q^2 + q + 1)(q^2 + 1)(q + 1)^2(q - 1)^2q^6.$$ Plugging this into the formula for 6-arcs in Theorem \ref{arcform} verifies our formula matches the one obtained by counting twisted cubics.

\subsection{Counting 7-arcs}

When $n =7$, we find six distinct non-isomorphic hyperfigurations. They are
\begin{align*}
h_1: \;&\mc{H} = \big\{\{0, 1, 2, 3\}, \{0, 1, 4, 5\}, \{0, 2, 4, 6\}, \{1, 2, 5, 6\}, \{1, 3, 4, 6\}, \{2, 3, 4, 5\}\big\},\\
&\mc{L} =  \big\{\big\}\\\\
h_2: \;&\mc{H}= \big\{\{0, 1, 2, 3\}, \{0, 1, 4, 5\}, \{0, 2, 4, 6\}, \{0, 3, 5, 6\}, \{1, 2, 5, 6\}, \{1, 3, 4, 6\}, \{2, 3, 4, 5\}\big\}, \\
&\mc{L}=\big\{\big\}\\\\
h_3:\; &\mc{H}= \big\{\{0, 1, 2, 3\}, \{0, 1, 2, 4\}, \{0, 1, 2, 5\}, \{0, 1, 2, 6\}, \{0, 3, 4, 5\}, \{1, 3, 4, 6\}, \{2, 3, 5, 6\}\big\},\\
&\mc{L}= \big\{\{0,1,2\}\big\}\\\\
h_4:\; &\mc{H} =\big\{\{0, 1, 2, 3, 4\}, \{0, 1, 2, 5\}, \{0, 1, 2, 6\}, \{0, 3, 4, 5\}, \{0, 3, 4, 6\},\{1, 3, 5, 6\}, \{2, 4, 5, 6\}\big\} ,\\
&\mc{L}=\big\{\{0, 1, 2\}, \{0,3, 4\}\big\}\\\\
h_5:\;&\mc{H}=\big\{\{0, 1, 2, 3, 4\}, \{0, 1, 2, 3, 5\}, \{0, 1, 2, 3, 6\}, \{0, 4, 5, 6\}, \{1, 4, 5, 6\}, \{2, 4, 5, 6\}, \{3, 4, 5,6 \}\big\},\\
&\mc{L}=\big\{\{0, 1, 2,3\}, \{4, 5, 6\}\big\}\\\\
h_6:\; &\mc{H}= \big\{\{0,1,2,3,4,5,6\}\big\},\\
&\mc{L}= \big\{\{0, 1, 2\}, \{0, 3, 4\}, \{0, 5, 6\}, \{1, 3, 5\}, \{1, 4, 6\}, \{2, 3, 6\}, \{2, 4, 5\}\big\}\\
\end{align*}

The hyperfiguration $h_6$ can be thought of as a projection down to the Fano plane. That is, the hyperfiguration has one plane, namely $\{0,1,2,3,4,5,6\}$, and seven lines in this plane that form a Fano plane.

\begin{defn} \cite{glynn07}
We say that $(P, B)$ is a $(n_k)$ configuration in $(k-1)$-dimensional projective space if every point lies on $k$ blocks (hyperplanes) and every block contains $k$ points.
\end{defn}

 This is a non-standard definition of a $(n_k)$ configuration as defined by Glynn \cite{glynn07}. Glynn uses this definition to define the complement of a configuration.
 
\begin{defn}\cite{glynn07}
	Let $(P, B)$ denote an $(n_k)$ configuration in $(k-1)$-dimensional projective space where $P$ represents a set of points and $B$ represents a set of blocks (hyperplanes). Let $B = \{b_1, \ldots, b_n\}$. Define a new $(n_{n-k})$ configuration $(P, B')$ in $(n-k-1)$-dimensional projective space where for each $b_i' \in B$, we say $p \in b_i'$ if and only if $p \not \in b_i$. This is called the \emph{complement} of $(P, B)$.
\end{defn}

\begin{rmk}
	The hyperfiguration $h_2$ is the complement of the Fano plane.
	There is a one-to-one correspondence between the strong realizations of the Fano plane in $\mathbb{P}^2(\Fq)$ and the strong realizations of $h_2$ in $\mb{P}^3(\Fq)$ modulo the collineation group 
	of $\mathbb{P}^2(\Fq)$ and $\mathbb{P}^3(\Fq)$ respectively. 
\end{rmk}

\begin{thm}\label{alg_thm7}
	The number of 7-arcs is given by
	\begin{align*}
		C_{7,4}(q) &= q^{21} - 28q^{20}+ 322q^{19} - 1925q^{18} + 5571q^{17}\\
		& + 839q^{16} - 18320q^{15} - 2695q^{14} + 7455q^{13} + 19111q^{12}\\
		& + 17074q^{11} - 9540q^{10} - 13027q^9 - 19922q^8 + 924q^7\\
		& + 14160q^6 +  \left(595q^3 - 8260q^2 + 20160q - 8820\right) \cdot A_{6}(4,q)\\
		& + 210A_{h_1}(4,q) + 180A_{h_2}(4,q) - 2520A_{h_3}(4,q) + 3780A_{h_5}(4,q).
	\end{align*}
\end{thm}

In order to understand the behavior of $C_{7,4}(q)$ as a function of $q$, we must understand the number of strong realizations of each hyperfiguration. Recall that we can assign a $k\times n$ generator matrix to each $n$-arc in $\mb{P}^{k-1}(\Fq)$ by assigning an affine representative of each point to each column. This generator matrix has the property that no $k \times k$ minor vanishes. Similarly, we can set up a $4 \times n$ generator matrix for each strong realization of a planar space in $\mb{P}^3(\F_q)$. In this case, any four points lie on a plane if and only if the $4 \times 4$ minor formed by these four points is equal to 0 in $\F_q$. More generally, any $\ell\ge 4$ points lie on a plane if and only if all $4 \times 4$ minors formed by the 4-subsets of these points are equal to 0 in $\F_q$. Similarly, any $\ell$ points lie on a line if and only if the $4 \times \ell$ matrix whose columns are these $\ell$ points does not have full rank. In other words, three points lie on a line if and only if all $3\times 3$ minors of the corresponding $4 \times 3$ matrix simultaneously vanish. 

A strong realization of a planar space $f$ then is a solution to the simultaneous vanishing of all minors corresponding to lines and planes in $f$ so that no additional lines and planes are formed.

Observe that all hyperfigurations $h_i$ on at most seven points contain a plane with exactly four points. Without loss of generality, we can set this plane equal to the plane $\{x=0\}$.

Given five general points in $f$, there exists a unique element in $\PGL_4(\Fq)$ that sends these five points to the points $[1:0:0:0], [0:1:0:0], [0:0:1:0], [0:0:0:1],$ and $[1:1:1:1]$ in $\mb{P}^3(\Fq)$. Observe that the points $[0:1:0:0], [0:0:1:0],$ and $[0:0:0:1]$ determine the plane $\{x=0\}$. 

These observations together lead to the following proposition.

\begin{prop}\label{strong_real_alg}
	Suppose that $h$ is a hyperfiguration on 7 points so that its first five points are in general position and $\{1, 2, 3, 5\}$ is a plane in $h$ containing 4 points. Let $$
	M_h = \begin{pmatrix}
		1&0&0&0&1&0&1\\
		0&1&0&0&1&y_1&y_2\\
		0&0&1&0&1&z_1&z_2\\
		0&0&0&1&1&w_1&w_2
		\end{pmatrix}.
			$$	
	Let $V_h$ be the variety defined by all polynomials formed by the vanishing of all $4 \times 4$ minors corresponding to planes in $h$ and $3\times 3$ minors corresponding to lines in $h$. Let $W_h$ be the open subset of $V_h$ for which all other $4 \times 4$ minors do not vanish and all $4 \times 3$ submatrices not corresponding to lines in $h$ have full rank. Then
	$$
	A_h(4,q) = \frac{|\PGL_4(\F_q)|}{(q-1)} \cdot \#W_h(\mb{F}_q).
	$$
\end{prop}

Proposition \ref{strong_real_alg} provides a method for computing $A_h(4,q)$ for the hyperfigurations on 7 points provided that the hyperfiguration contains five points in general position. Of course, a strong realization of hyperfiguration $h_i$ may not have the points $\{0,1,2,3,4\}$ in general position or may not contain the 4-point plane $\{1,2,3,5\}$. Thus the columns of a generator matrix for $h_i$ will be a permutation of the columns of $M_h$ up to rescaling each column and possibly also permuting the indices of the variables.


\begin{prop}\label{strong_real7}
Let $$a(q) = \begin{cases}1 & q \equiv 0 \pmod{2}\\0 & q \equiv 1 \pmod{2}\end{cases}.$$	The number of strong realizations for each hyperfiguration is given by

	\begin{align*}
	A_{h_1}(4,q)&=(1-a(q)) \cdot |\PGL_4(\F_q) |\\\\
	A_{h_2}(4,q) &= a(q)  \cdot |\PGL_4(\F_q)|\\\\
		A_{h_3}(3,q) &= (q-2) \cdot |\PGL_4(\F_q)|\\\\
	A_{h_4}(4,q)&=  |\PGL_4(\F_q)| \\\\
	A_{h_5}(4,q) &= (q^2 + q + 1)(q^2 + 1)(q + 1)^2(q - 1)^2(q - 2)q^6\\\\
	A_{h_6}(4,q) &= a(q)\cdot q \cdot (q-1) \cdot (q-2) \cdot |\PGL_3(\F_q)|.\\
	\end{align*}
\end{prop}
\begin{proof} We consider each hyperfiguration $h_i$ separately.\\
	
\noindent\textbf{Hyperfiguration $h_1$:} The points 0, 3, 4, 5, and 6 are in general position. Since the plane $\{2,3,4,5\}$ contains exactly four points, we can fix this plane to be $x=0$. This gives the following generator matrix
	$$M_{h_1}=
	\begin{pmatrix}
		1 &1 &0&0&0&0&1\\
		0 & y_1&y_2&1&0&0&1\\
		0 &z_1&z_2&0&1&0&1\\
		0&w_1&w_2&0&0&1&1
	\end{pmatrix},
	$$
	which is a permutation of the matrix $M_h$ from Proposition \ref{strong_real_alg}.
	By computing the determinants corresponding to the planes in $h_1$, we obtain a variety $V_{h_1}$ defined by the polynomials
	$$\begin{cases}
		&-w_1+1\\
		&y_2z_1	-y_1z_2-y_2+z_2\\
		&y_1\\
		&y_2-w_2\\
		&-z_2w_1 + z_1w_2.
	\end{cases}$$
	Substituting shows that we can understand the number of $\Fq$-points on $V_{h_1}$ by understanding the solutions to $$2z_1 w_2 - w_2 =0.$$
	
	We then compute $W_{h_1}$, the open subset of $V_{h_1}$ which disallows additional collinearities or coplanarities in the realization of $h_1$. We find that $W_{h_1}$ is defined by the vanishing of the polynomials defining $V_{h_1}$ together with the following inequalities
	
	\begin{equation}\label{h1_ineqs}
	\begin{cases}
		&z_1 \ne 0,1\\
		&w_2 \ne 0.
	\end{cases}
	\end{equation}
	
	If the characteristic of $\F_q$ is even, then $w_2=0$, which is impossible by \eqref{h1_ineqs}. If the characteristic of $\Fq$ is odd, we have two cases: either $w_2=0$ or $z_1 = 2^{-1}$. Since $w_2 \ne 0$ by \eqref{h1_ineqs}, then we must have $z_1 = 2^{-1}$.\\

\noindent\textbf{Hyperfiguration $h_2$:} As remarked above, this hyperfiguration is the complement of the Fano plane in $\mb{P}^2(\Fq)$. Further, there is a one-to-one correspondence between strong realizations of the Fano plane and $h_2$ modulo their collineation groups. We derive $A_{h_2}(4,q)$ from the number of strong realizations of the Fano plane in $\mb{P}^2(\Fq)$.\\

\noindent\textbf{Hyperfiguration $h_3$:} As before we set up a generator matrix 
$$M_{h_3}=
\begin{pmatrix}
	0 &0&0&1&0&1&1\\
	y_1 & 1&0&y_2&0&0&1\\
	z_1 &0&1&z_2&0&0&1\\
	w_1&0&0&w_2&1&0&1
\end{pmatrix}
$$
and consider the vanishing of all minors corresponding to planes and lines in $h_3$. We arrive at the following set of equations
$$\begin{cases}
	&w_1=0\\
	&z_2=1\\
	&y_2=w_2\\
	&y_1=y_2z_1.
\end{cases}$$
We have the inequalities
$$
\begin{cases}
	&w_2 \ne 0,1\\
	&z_1 \ne 0.
\end{cases}
$$
Thus any choice of $z_1 \ne 0$ and $w_2 \ne 0,1$ gives a strong realization of $h_3$.\\

\noindent\textbf{Hyperfiguration $h_4$:} The points 0, 1, 4, 5 and 6 are in general position. We set up the matrix
	$$M_{h_4}=
	\begin{pmatrix}
		1& 1 & 0 &1 & 0&0&0\\
		0&1&y_1&y_2&1&0&0\\
		0&1&z_1 &z_2&0&1&0\\
		0&1&w_1&w_2&0&0&1
	\end{pmatrix}.
	$$
	Computing the determinants that correspond to the planes and lines in $h_4$ gives the equations 
	$$
	\begin{cases}
		&y_1=z_1=w_1\\
		&w_2=z_2 =0\\
		&y_2 = 1.
	\end{cases}
$$
Further, the inequalities reduce to
$$
w_1 \ne 0.
$$
Once we choose a value for $w_1$, every other variable is determined.\\ 

\noindent\textbf{Hyperfiguration $h_5$:} We can compute $A_{h_5}(4,q)$ by counting. This is computed similarly to that of the hyperfiguration on 6 points.\\
	
\noindent\textbf{Hyperfiguration $h_6$:} It is well-known that the number of strong realizations of the Fano plane in $\mb{P}^2(\F_q)$ is $|\PGL_3(\F_q)| \cdot a(q)$. We can fix the embedding of the Fano plane into $\mb{P}^2(\F_q)$ given by
$$
\begin{pmatrix}
	0 & 0 & 0 &1&1&1&1\\
	1&0&1&0&1&0&1\\
	0&1&1&0&0&1&1
\end{pmatrix}.
$$
In order to determine the number of strong realizations of $h_6$ in $\mb{P}^3(\F_q)$, we add an additional coordinate to each of the points in $\mb{P}^2(\F_q)$ as follows
$$
\begin{pmatrix}
	x_1 & x_2 & x_3 &x_4 &x_5&x_6&x_7\\
	0 & 0 & 0 &1&1&1&1\\
	1&0&1&0&1&0&1\\
	0&1&1&0&0&1&1
\end{pmatrix}.
$$
Clearly there are no realizations when $q$ is odd, so we assume $q$ is even.
Since all points must lie in a single plane, all $4 \times 4$ minors must vanish. Solving this system when $q$ is even gives
$$
\begin{cases}
	x_1 =x_6 + x_7\\
	x_2 = x_5 + x_7\\
	x_3 = x_5 + x_6\\
	x_4 = x_5 + x_6 + x_7.
\end{cases}
$$
Further, we get the following inequalities:
$$\begin{cases}
	x_5 &\ne x_6 + x_7\\
	x_5 &\ne x_7\\
	x_6 &\ne x_7.
\end{cases}$$

We can choose $x_7 \in \F_q$. We then select $x_6 \ne x_7$, and $x_5 \ne x_6 + x_7$ or $x_7$. There are $q(q-1)(q-2)$ such choices. Once these values are chosen, the variables $x_i$ for $1\le i \le 4$ are fixed.
\end{proof}

Theorem \ref{arcform} follows from Theorem \ref{alg_thm7} and Proposition \ref{strong_real7}. Notice that $C_{7,4}(q)$ is a quasipolynomial in $q$. 

\subsection{Using this strategy to count $C_{n,4}(q)$ for larger $n$}
Table \ref{psf_count} shows that the number of planar spaces likely grows exponentially in $n$. Further, the number of hyperfigurations on $n$ points grows quickly.

Therefore is seems that this strategy for counting $n$-arcs will quickly become infeasible. In future work, we intend to study 8-arcs in $\mb{P}^3(\F_q)$. While it is likely time-consuming to compute $C_{8,4}(q)$ exactly, we will determine whether or not this counting function is a quasipolynomial in $q$.

\section{Generalizing to higher-dimensional projective space}\label{higher_dim_sec}

In order to generalize the ideas in this paper to produce a simpler expression for $C_{n,k}(q)$, we must understand what objects generalize planar spaces and hyperfigurations. There is a natural generalization of planar space. 

\begin{defn}
	For $k \ge 3$, a $k$-planar space is a tuple $(\mc{P}, \mc{H}_1, \mc{H}_2, \ldots , \mc{H}_{k-2})$ where $\mc{H}_i \subset 2^{\mc{P}}$ for each $1\le i \le k-2$ and every $(i+1)$ distinct points that do not lie in a subset $H \subset \mc{H}_j$ for $1 \le j < i$ form a unique subspace in $\mc{H}_i$. Observe that this is a $(k-1)$-dimensional space.
\end{defn}

Recall that the proof of Theorem \ref{p3count} follows from induction by using Lemma \ref{mainlemma} to express realizations of non-hyperfigurations on $n$ points in terms of realizations of hyperfigurations on at most $n-1$ points. The definition for hyperfiguration was derived by simply considering all planar spaces for which Lemma \ref{mainlemma} did not apply. Our generalization of a hyperfiguration will be defined similarly: we will determine for which indices in $f$ it is possible to prove that $\mu(f,g)$ is a polynomial in $q$.

\begin{defn}
	The \emph{index} of a point $m$ of a $k$-planar space is given by $(i_1, i_2, \ldots, i_{k-2})$ where $i_j$ is the number of $(k-1-j)$-dimensional objects incident with point $m$.
\end{defn}

\begin{defn}
	A \emph{k-hyperfiguration} is a $k$-planar space for which the index of every point satisfies $i_j > k-j$ for some $1 \le j \le k-2$ or the index is in a finite set of surprising indices.
\end{defn}

It is natural to wonder how many surprising indices there are for each $k > 4$. If the number of surprising indices grows too quickly, it is possible that most $k$-planar spaces are $k$-hyperfigurations. We demonstrate upper and lower bounds on the size of the set of surprising indices. 

\begin{prop}\label{onetwoline}
	Let $f$  be a $k$-planar space and let $m$ be a point with index $(i_1, i_2, \ldots, i_{k-2})$ such that $0 < i_{k-2} \le 2$. Then for any $g \ge f'$, $\mu(f,g)$ is a polynomial in $q$.
\end{prop}
\begin{proof}
	Recall that $i_{k-2}$ denotes the number of lines incident with the point $m$. Fix some $g \ge f'$ and let $P_g$ be a strong realization of $g$.
	
	If $i_{k-2} = 1$, let $L$ be the line in $f$ containing $m$ and let $L'$ be the corresponding line in $f'$. For every set $H'$ in $f'$ corresponding to the $a$-dimensional object in $f$ containing $m$, we must have $L' \subset H'$. Extend $L'$ to $L_g$ in $g$. For every extension $H_g$ of $H'$ in $g$, we must have $L_g \subset H_g$. Thus adding any point to $L_g$ gives a weak realization of $f$, so
	$$
	\mu(f,g) = q+1 - \#(P_g \cap L_g).
	$$
	
	If $i_{k-2} = 2$, then the point $m$ lies on the intersection of two lines $L_1$ and $L_2$ in $f$. The corresponding lines $L_{1,g}$ and $L_{2,g}$ must intersect in $g$. It suffices to add their intersection point $r$ to $g$ to obtain a weak realization of $f$. Thus
	$$
	\mu(f,g) = \begin{cases}
		0 & r \in P_g\\
		1& r \not \in P_g
		\end{cases}.
	$$
\end{proof}

\begin{prop}\label{lb_surprising_index}
	Let $k > 3$ and let $S_{k}$ be the number of surprising indices in $\mb{P}^{k-1}(\Fq)$. Then
	$$
	S_k \ge (k-1)S_{k-1}.
	$$
\end{prop}
\begin{proof}
Pick any surprising index $(i_1, \ldots, i_{k-1})$ in $(k-2)$-dimensional space. Then for any $i_0 \ge 1$, we claim $(i_0, i_1, \ldots, i_{k-1})$ is a surprising index in $(k-1)$-dimensional space. Let $f$ be a $k$-planar space which has a point of index $(i_0, i_1, \ldots, i_k)$ such that $i_0 \ge 1$. It is possible that $a$-dimensional objects for $0 \le a \le k-3$ lie in a single hyperplane of $f$. In this case, it is impossible to determine $\mu(f,g)$ since we reduce to studying objects in a $(k-1)$-planar space.
\end{proof}

\begin{cor}\label{up_low_bds}
	Let $S_k$ be the number of surprising indices in $\mb{P}^{k-1}(\Fq)$. Then
	$$
	\frac{(k-1)!}{6} \le S_k \le \frac{k!}{6}.
	$$
\end{cor}

There are a total of $\frac{k!}{2}$ potential indices for a point in a $k$-planar space (some of these indices will be impossible). Applying Corollary \ref{up_low_bds}, we find that the ratio of surprising indices to potential indices is between $\frac{1}{3k}$ and $\frac{1}{3}$. More work needs to be done to determine whether the ratio of surprising indices to potential indices will go to 0 or a non-zero constant.

\begin{example}
	Corollary \ref{up_low_bds} gives $1 \le S_4 \le 4$. When found that there is exactly one surprising index when $k=4$. 
	 Thus the ratio of surprising indices to potential indices is $\frac{1}{12}$. Recall that in Lemma \ref{mainlemma}, the cases (0,2) and (2,2) were impossible.
\end{example}

\begin{rmk}
	The number of surprising indices grows very quickly. Thus the definition of $k$-hyperfiguration given in this section likely needs refinement. Observe that in Lemma \ref{mainlemma}, we omit an index $I$ if there exists a planar space $f$ with index $I$ and a planar space $g \ge f'$ such that $\mu(f,g)$ cannot be explicitly given as a polynomial in $q$. However, it is possible that for some planar spaces $f$ with index $I$ and every $g \ge f'$, we can conclude that $\mu(f,g)$ is a polynomial in $q$. A refinement of Lemma \ref{mainlemma} could be made that checks more than just the index of the point to be removed. The number of surprising indices would remain the same, but the number of $k$-hyperfigurations would likely decrease. 
\end{rmk}

\begin{defn}
	A \emph{strong realization} of a $k$-hyperfiguration $h$ is an injective mapping $\sigma: \mc{P} \rightarrow \mb{P}^{k-1}(\F_q)$ such that for all subsets $Q \subseteq \mc{P}$ and all $1 \le a \le k-2$, $Q$ is contained in an $a$-dimensional subset of $h$ if and only if $\sigma(Q)$ is contained in an $a$-dimensional subset of $\mb{P}^{k-1}(\Fq)$.
\end{defn}

\begin{thm}
	There exists polynomials $p(q)$ and $p_h(q)$ for which
	$$
	C_{n,k}(q) = p(q) + \sum_{h} p_h(q) A_h(k,q)
	$$
	where the summation is over all isomorphism classes of $k$-hyperfigurations on at most $n$ points.
\end{thm}

\begin{proof}
	By definition of a $k$-hyperfiguration, we can show that if $f$ is not a $k$-hyperfiguration, then 
	$$
	B_f(k,q) =\sum_{g \ge f'}\mu(f,g) A_{g}(k,q) 
	$$
	for some polynomials $\mu(f,g)$. The proof is an inductive argument similar to that in the proof of Theorem \ref{p3count}.
\end{proof}


Of course, this is only an existence theorem -- for each $k$, one would need to understand how to compute $\mu(f,g)$ for every planar space $f$ and every $g \ge f'$. This becomes infeasible as $k$ grows as there are $\frac{k!}{2}$ potential indices to consider.

\section{Acknowledgments} 
This work was supported by the NSF grant DMS 1802281. The author thanks Nathan Kaplan for many invaluable conversations and Max Weinreich for helpful comments. The author also thanks the anonymous referee for comments and suggestions that have improved this paper. Parts of this work appeared in the author's Ph.D. thesis.

\bibliographystyle{habbrv}
\footnotesize{\bibliography{../../../../Bibliography/bib_all}}

\begin{thebibliography}{10}

\bibitem{balllavrauw}
S.~Ball and M.~Lavrauw.
\newblock Arcs in finite projective spaces.
\newblock {\em EMS Surv. Math. Sci.}, 6(1):133--172, 2020.

\bibitem{bdmp}
D.~{Bartoli}, A.~A. {Davydov}, S.~{Marcugini}, and F.~{Pambianco}.
\newblock {On planes through points off the twisted cubic in $\mathrm{PG}(3,q)$
  and multiple covering codes}, 2020, \url{https://arxiv.org/abs/1909.00207},
  arxiv:1909.00207.

\bibitem{bb}
L.~Batten and A.~Beutelspacher.
\newblock {\em The theory of finite linear spaces - combinatorics of points and
  lines}.
\newblock Cambridge University Press, 1993.

\bibitem{betten}
A.~Betten and D.~Betten.
\newblock Linear spaces with at most 12 points.
\newblock {\em J. Comb. Des.}, 7:119--145, 1999.

\bibitem{glynn}
D.~G. Glynn.
\newblock Rings of geometries {II}.
\newblock {\em J. Combin. Theory Ser. A}, 49(1):26 -- 66, 1988.

\bibitem{glynn07}
D.~G. Glynn.
\newblock A note on $n_k$ configurations and theorems in projective space.
\newblock {\em Bull. Aust. Math. Soc.}, 76(1):15–31, 2007.

\bibitem{hirschfeldb}
J.~W.~P. Hirschfeld.
\newblock The main conjecture for {MDS} codes.
\newblock In C.~Boyd, editor, {\em Cryptography and Coding}, pages 44--52.
  Springer Berlin Heidelberg, 1995.

\bibitem{hirschfelda}
J.~W.~P. Hirschfeld.
\newblock {\em Projective Geometries Over Finite Fields}.
\newblock Oxford Mathematical Monographs. The Clarendon Press, Oxford
  University Press, New York, second edition, 1998.

\bibitem{hirschfeldthas}
J.~W.~P. Hirschfeld and J.~Thas.
\newblock Open problems in finite projective spaces.
\newblock {\em Finite Fields Appl.}, 32:44--81, 03 2015.

\bibitem{iss}
A.~V. {Iampolskaia}, A.~N. {Skorobogatov}, and E.~A. {Sorokin}.
\newblock Formula for the number of [9,3] {MDS} codes.
\newblock {\em IEEE Trans. Inform. Theory}, 41(6):1667--1671, 1995.

\bibitem{kaipa}
K.~V. {Kaipa}.
\newblock An asymptotic formula in $q$ for the number of $[n, k]$~$q$-ary {MDS}
  codes.
\newblock {\em IEEE Trans. Inform. Theory}, 60(11):7047--7057, 2014.

\bibitem{kklpw}
N.~Kaplan, S.~Kimport, R.~Lawrence, L.~Peilen, and M.~Weinreich.
\newblock Counting arcs in the projective plane via {G}lynn's algorithm.
\newblock {\em J. Geom.}, 12 2016.

\bibitem{mmib}
Y.~Matsumoto, S.~Moriyama, H.~Imai, and D.~Bremner.
\newblock Matroid enumeration for incidence geometry.
\newblock {\em Discrete Comput. Geom.}, 47:17--43, 2012.

\bibitem{mnev}
N.~E. Mnëv.
\newblock The universality theorems on the classification problem of
  configuration varieties and convex polytopes varieties.
\newblock In O.~Y. Viro and A.~M. Vershik, editors, {\em Topology and Geometry
  --- Rohlin Seminar}, pages 527--543. Springer Berlin Heidelberg, Berlin,
  Heidelberg, 1988.

\bibitem{moorhouse}
C.~E. Moorhouse.
\newblock Lecture notes in incidence geometry.
\newblock \url{http://ericmoorhouse.org/handouts/Incidence$_Geometry.pdf},
  August 2007.

\bibitem{paksoffer}
I.~{Pak} and A.~{Soffer}.
\newblock {On Higman's $k(U_n(\mathbb{F}_q))$ conjecture}, 2015,
  \url{https://arxiv.org/abs/1507.00411}, arXiv:1507.00411.

\bibitem{segre}
B.~Segre.
\newblock Curve razionali normali ek-archi negli spazi finiti.
\newblock {\em Ann. Mat. Pura Appl.}, 39:357--379, 1955.

\bibitem{skorobogatov_matroid}
A.~N. Skorobogatov.
\newblock On the number of representations of matroids over finite fields.
\newblock {\em Des. Codes Cryptography}, 9(2):215–226, Oct. 1996.

\bibitem{sage}
W.~Stein et~al.
\newblock {\em {S}age {M}athematics {S}oftware ({V}ersion 8.9)}.
\newblock The Sage Development Team.
\newblock {\tt http://www.sagemath.org}.

\bibitem{vaughanlee}
M.~Vaughan-Lee.
\newblock Graham {H}igman's {PORC} theorem.
\newblock {\em Jahresber. Dtsch. Math. Ver.}, 114(89), 2012.

\end{thebibliography}


\begin{thebibliography}{1}
\bibitem{batten2} L. M. Batten, Combinatorics of Finite Geometries. Second Edition. Cambridge University Press, Cambridge, 1997.

\bibitem{batten} L. M. Batten and A. Beutelspacher, The Theory of Finite Linear Spaces: Combinatorics of Points and Lines. Cambridge University Press, Cambridge, 2009.

\bibitem{delandtsheer} A. Delandtsheer, Finite (Line-Plane)-Flag-Transitive Planar Spaces. Geometriae Dedicata 41:145 153, 1992.

\bibitem{dsvl} M. du Sautoy and M. Vaughan-Lee, Non-PORC behaviour of a class of descendant p-groups,
J. Algebra 361 (2012), 287–312.

\bibitem{glynn} D. Glynn, Rings of geometries. II. J. Combin. Theory Ser. A 49 (1988), no. 1, 26-66.

\bibitem{grunbaum2009} B. Gr\"unbaum, Configurations of Points and Lines. Graduate Studies in Mathematics, 103. American Mathematical Society, Providence, RI, 2009.

\bibitem{hirschfeldthas} J. W. P. Hirschfeld and J. A. Thas, Open problems in finite projective spaces. Finite Fields Appl. 32 (2015), 44-81.

\bibitem{isham et al} K. Isham, N. Kaplan, R. Lawrence, and M. Weinreich, The number of 10-arcs in a projective plane is not quasipolynomial. Preprint.

\bibitem{iss} A. Iampolskaia, A. N. Skorobogatov, and E. Sorokin, Formula for the number of $[9,3]$ MDS codes. IEEE Trans. Inform. Theory 41 (1995), no. 6, part 1, 1667-1671.

\bibitem{kaplan et al} N. Kaplan, S. Kimport, R. Lawrence, L. Peilen, and M. Weinreich, Counting arcs in projective planes via Glynn's algorithm, in J. Geom 108 (2017), no. 3, 1013--1029.

\bibitem{paksoffer} I. Pak and A. Soffer, On Higman's $k(U_n(\Fq))$ Conjecture. Preprint (2015). arXiv:1507.00411

\bibitem{rolland1993} R. Rolland and A. N. Skorobogatov, D\'enombrement des configurations dans le plan projectif, in Proc. Arithmetic, geometry, and coding theory (Luminy, 1993), 199-207, de Gruyter, Berlin, 1996.

\bibitem{Skorobogatov} A. N. Skorobogatov, Linear codes, strata of Grassmannians, and the problems of Segre, in Lecture Notes in Mathematics, no. 1518 (1992), Coding Theory and Algebraic Geometry 1991, H. Stichtenoth, M. A. Tsfasman, Eds., pp. 210-223.

\bibitem{stanley} R. Stanley, Enumerative Combinatorics, 2nd Ed.  Cambridge University Press. 2011.

\bibitem{sturmfels} B. Sturmfels, Computational algebraic geometry of projective configurations. J. Symbolic Computation (1991) 11, 595-618.

\bibitem{vaughanlee} M. Vaughan-Lee, Graham Higman's PORC Conjecture. Jahresber. Dtsch. Math. Ver. (2012) 114: 89.
\end{thebibliography}

\end{document}